\documentclass[11pt,a4paper]{amsart}
\usepackage[pdftex]{graphicx}
\usepackage[pdftex]{xcolor}

\usepackage{mathtools}

\makeatletter
\newtheorem*{rep@theorem}{\rep@title}
\newcommand{\newreptheorem}[2]{%
\newenvironment{rep#1}[1]{%
 \def\rep@title{#2 \ref{##1}}%
 \begin{rep@theorem}}%
 {\end{rep@theorem}}}
\makeatother

\definecolor{RedOrange}{cmyk}{ 0, 0.77, 0.87, 0}
\definecolor{RoyalPurple}{cmyk}{ 0.84, 0.53, 0, 0}
\definecolor{YellowGreen}{cmyk}{ 0.44, 0, 0.74, 0}
\definecolor{Fuchsia}{cmyk}{ 0.47, 0.91, 0, 0.08}
\definecolor{Blue}{cmyk}{ 0.84, 0.53, 0, 0}
\definecolor{BlueViolet}{cmyk}{ 0.84, 0.53, 0, 0}
\definecolor{Black}{cmyk}{ 0.75, 0.68, 0.67, 0.9}
\usepackage{verbatim}
\usepackage{amsmath}
\usepackage{amssymb, mathrsfs}
\usepackage{amsbsy}

\usepackage{amscd}
\usepackage{amsthm}
\usepackage[english]{babel}
\usepackage{todonotes}
\usepackage[T1]{fontenc}

\usepackage{color}

\newcommand{\R}{\mathbb{R}}
\newcommand{\B}{\mathbb{B}}

\newcommand{\N}{\mathbb{N}}
\newcommand{\e}{\varepsilon}

\newcommand{\E}{\mathbb{E}}
\newcommand{\Z}{\mathbb{Z}}

\renewcommand{\P}{\mathbb{P}}

\newcommand{\lin}{\left[\kern-0.15em\left[}
\newcommand{\rin} {\right]\kern-0.15em\right]}
\newcommand{\linf}{[\kern-0.15em [}
\newcommand{\rinf} {]\kern-0.15em ]}
\newcommand{\ilin}{\left]\kern-0.15em\left]}
\newcommand{\irin} {\right[\kern-0.15em\right[}

\usepackage{constants}

\newconstantfamily{c}{symbol=c}
\newconstantfamily{a}{symbol=\alpha}
\newcommand{\secno}[1]{\thesection.\arabic{#1}}
\newconstantfamily{kE}{
symbol=\mathcal{E},
format=\secno,
reset={section}
}
\renewcommand{\tilde}{\widetilde}

\newtheorem{lem}{Lemma}[section]
\newtheorem{remark}[lem]{Remark}

\newtheorem{thm}[lem]{Theorem}

\newtheorem {Def}[lem] {Definition}

\usepackage{color}
\definecolor{lilas}{RGB}{182, 102, 210}

\numberwithin{equation}{section}

\title[Divergence of shape fluctuation]
{Divergence of shape fluctuation for general distributions in
First-passage percolation}
\date{\today}
\author{Shuta Nakajima} 
\address[Shuta Nakajima]
{Graduate School of Mathematics, University Nagoya.}
\email{njima@math.nagoya-u.ac.jp}


\begin{document}
\maketitle

\begin{abstract}
  We study the shape fluctuation in the first-passage percolation on $\Z^d$. It is known that it diverges when the distribution obeys Bernoulli in [Yu Zhang. The divergence of fluctuations for shape in
first passage percolation. {\em Probab. Theory. Related. Fields.} 136(2)
298--320, 2006]. In this paper, we extend the result to general distributions. 
\end{abstract}

\section{Introduction}
First-passage percolation is a random growth model, which was first introduced by Hammersley and Welsh in 1965. The model is defined as follows. The vertices are the elements of $\Z^d$. Let us denote the set edges by $E^d$:
$$E^d=\{\{ v,w \}|~v,w\in\Z^d,~|v-w|_1=1\},$$
where we set $|v-w|_1=\sum^d_{i=1}|v_i-w_i|$ for $v=(v_1,\cdots,v_d)$, $w=(w_1,\cdots,w_d)$. Note that we consider non-oriented edges in this paper, i.e., $\{ v,w \}=\{ w,v \}$ and we sometimes regard $\{ v,w \}$ as a subset of $\Z^d$ with a slight abuse of notation. We assign a non-negative random variable $\tau_e$ on each edge $e\in E^d$, called the passage time of the edge $e$. The collection $\tau=\{\tau_e\}_{e\in E^d}$ is assumed to be independent and identically distributed with common distribution $F$. \\
    
  A path $\gamma$ is a finite sequence of vertices $(x_1,\cdots,x_l)\subset\Z^d$ such that for any $i\in\{1,\cdots,l-1\}$, $\{x_i,x_{i+1}\}\in E^d$. It is customary to regard a path as a subset of edges as follows: given an edge $e\in E^d$, we write $e\in \gamma$ if there exists $i\in \{1\cdots,l-1\}$ such that $e=\{x_{i},x_{i+1}\}$.\\

  Given a path $\gamma$, we define the passage time of $\gamma$ as
$$T(\gamma)=\sum_{e\in\gamma}\tau_e.$$
For $x\in\R^d$, we set $[x]=([x_1],\cdots,[x_d])$  where $[a]$ is the greatest integer less than or equal to $a$. Given two vertices $v,w\in\R^d$, we define the {\em first passage time} between vertices $v$ and $w$ as
$$T(v,w)=\inf_{\gamma:[v]\to [w]}T(\gamma),$$
where the infimum is taken over all finite paths $\gamma$ starting at $[v]$ and ending at $[w]$. A path $\gamma$ from $v$ to $w$ is said to be {\em optimal} if it attains the first passage time, i.e., $T(\gamma)=T(v,w)$.\\ 

By Kingman's subadditive ergodic theorem, if $\E \tau_e<\infty$, for any $x\in\R^d$, there exists a non-random constant $g(x)\ge 0$ such that

\begin{align}\label{kingman}
  g(x)=\lim_{t\to\infty}t^{-1} T(0,t x)=\lim_{t\to\infty}t^{-1} \E[T(0,t x)]\hspace{4mm}a.s.
\end{align}
This $g(x)$ is called the {\em time constant}. Note that, by subadditivity, if $x\in\Z^d$, then $g(x)\le \E T(0,x)$ and moreover for any $x\in\R^d$, $g(x)\le \E T(0,x)+2d\E \tau_e$. It is easy to check homogeneity and convexity: $g(\lambda x)=\lambda g(x)$ and $g(r x+(1-r)y)\le r g(x)+(1-r)g(y)$ for $\lambda\in\R$, $r\in[0,1]$ and $x,y\in\R^d$. It is well-known that if $F(0)<p_c(d)$, then $g(x)>0$ for any $x\neq 0$ \cite{Kes86}. Therefore, if $F(0)<p_c(d)$, then $g:\R^d\to \R_{\geq 0}$ is a norm.\\
  
 \subsection{Background and related works}\label{subsec:his}
 
   %  \end{comment}
    We define $B(t)=\{x\in\R^d|~T(0,x)\leq t\}$ as the fluid region starting from the origin at time $t$. Let $\B_d=\{x\in\R^d|~g(x)\leq 1\}.$ Cox and Durrett proved the following shape theorem~\cite{CD81}: If $F(0)<p_c(d)$ and $\E\tau_e<\infty$, for any $\e>0$,
\begin{equation}\label{shape-thm1}
  \lim_{t\to\infty}\P\left(t(1-\e)\B_d \subset B(t)\subset t(1+\e)\B_d)\right)=1.
\end{equation}
 Since the result of \eqref{shape-thm1} corresponds to the law of large number of $B(t)$, the next step is to consider the rate of the convergence, that is the minimum value $f(t)$ satisfying $(t-f(t))\B_d \subset B(t)\subset (t+f(t))\B_d$, which is called the shape fluctuation denoted by ${\rm F}(B(t),t\B_d)$ (we will extend the definition to more general forms in Definition \ref{Def: fluc}).\\

  %One of the objects of interest in FPP is the asymptotic growth of the boundary of $B(t)$. It is widely believed that the suitable scaling limit of the boundary is the solution of KPZ-equation if $d=2$ \cite{KPZ86}. Although some other models have been found so that the above conjecture holds (e.g., \cite{BG97}), there are no models of FPP that are shown mathematicaly. From the KPZ-theory, it is conjectured that if $d=2$, the order of the fluctuation of the shape is $t^{1/3}$ (see \cite{KS91}).\\
  
   Due to the works of Kesten \cite{Kes93} and Alexander \cite{Alex97}, the shape fluctuation is $O((t\log{t})^{1/2})$ for any dimension. The first attempt for the lower bound was due to Pemantle and Peres \cite{PP94} where they proved that if $F$ is exponential distribution and $d=2$, then the shape fluctuation diverges. Thereafter, Chatterjee \cite{Chatt17}  proved that under mild smoothness and decay assumptions on the edge weight distribution, the shape fluctuation grows at least $t^{1/8-o(1)}$ for $d=2$.\\

   On the other hand, these problems also have interesting features in higher dimensions. Some physicists predicted that if $d$ is sufficiently large, the fluctuation does not diverge in some sense. See the introduction of \cite{NP95}. The scaling limits in higher dimensions are controversial issues even in physics and there are some candidates. See \cite{AJ15} and references therein.  However, Zhang showed that if $\tau$ obeys the Bernoulli distribution, the shape fluctuation diverges \cite{Zhang06}. Indeed he showed that for any sufficiently small $c>0$, there exists $C>0$ such that for any $\Gamma\subset \R^d$,
\begin{figure}[b]
  \scalebox{0.55}{\includegraphics[width=7.0cm]{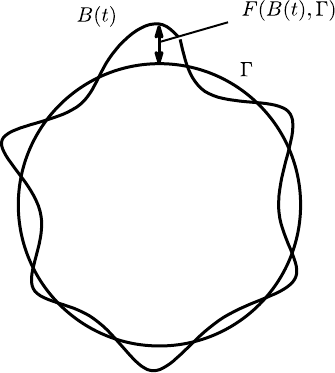}}\hspace{6mm}
  \scalebox{0.55}{\includegraphics[width=8.0cm]{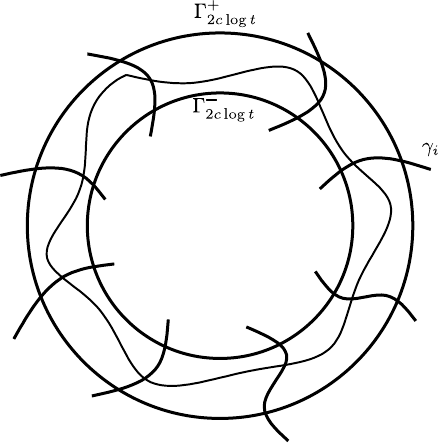}}
  \caption{}
  \label{fig:shape-fluc}
  Schematic picture of the shape fluctuation.
  \label{fig:two}
\end{figure}
 \begin{equation}\label{Zhang:shape}
   \P({\rm F}(B(t),\Gamma)\le c\log t)\leq  Ct^{-d+2-2c\log{p}},
 \end{equation}
 where $p=\P(\tau_e=0)=1-\P(\tau_e=1)$. Note that the bound is meaningful only when $d>2$. (Although \eqref{Zhang:shape} is stated without any restriction to $\Gamma$, it seems that a certain natural restriction such as convexity is required as in Theorem \ref{thm:shape}.)  His method relies on Russo's formula and it seems not easily extended directly to general distributions. In this paper, a different approach is taken to overcome this problem. Indeed, we apply a variant of the resampling argument introduced by van den Berg and Kesten \cite{BK93} and use it inductively to get the stretched-exponential bound. As a result, we prove the statement not only for general useful distributions but also a stronger estimate. It is worth noting that our model includes the  Eden or Richardson model.\\

 We consider the fluctuation from general convex sets following \cite{Zhang06}.
\begin{Def}\label{Def: fluc}
   For $l>0$ and a subset $\Gamma$ of $\R^d$, let
   $$\Gamma^-_l=\{v\in\Gamma|~d(v,\Gamma^c)\geq{}l\}\text{ and }\Gamma^+_l=\{v\in\R^d|~d(v,\Gamma)\leq{}l\},$$
   where $d$ is the Euclidean distance. Given three sets $A,B,C\subset \R^d$, we define the fluctuation of $A$ from $B$ inside $C$ as
    $${\rm F}_C(A,B)=\inf\{\delta>0|~B^-_{\delta}\cap C\subset A\cap C \subset B^+_{\delta}\cap C\}.$$
\end{Def}
\begin{remark}
  If $A,B,C$ are convex subsets, ${\rm F}_C(A,B)$ coincides with the Hausdorff distance $d_H(A\cap C,B\cap C)$. Although they do not coincide in general, the same proof still works with a suitable modification and the results below hold even when we replace ${\rm F}_C(A,B)$ by $d_H(A\cap C,B\cap C)$. 
  \end{remark}
When $A=B(t),B=t\B_d,C=\R^d$, the fluctuation ${\rm F}_{\R^d}(B(t),t\B_d)$ is simply the shape fluctuation ${\rm F}(B(t),t\B_d)$ mentioned above. To consider the directional shape fluctuation, we define the following cone. 
  \begin{Def}
    Given $\theta\in \R^d$ and $r>0$, let
    $$L(\theta,r)=\{a\cdot \mathbf{v}|~a\in[0,\infty),~\mathbf{v}\in B(\theta,r)\},$$
      where $B(x,r)$ is the closed ball whose center is $x$ and radius is $r$. 
  \end{Def}
  Note that if $r>2$, for any $\theta\in \mathbb{S}^{d-1}=\{x\in\R^d|~|x|=1\}$, $L(\theta,r)$ is the entire $\R^d$.
    We restrict ourselves to the following class of distributions. A distribution $F$ is said to be {\em useful} if  
   \begin{equation}\label{Def:useful}
     \P(\tau_e=F^-)<
    \begin{cases}
    p_c(d) & \text{if $F^-=0$} \\
    \vec{p}_c(d)& \text{otherwise},
    \end{cases}
    \end{equation}
   where $p_c(d)$ and $\vec{p}_c(d)$ stand for the critical probabilities for $d$-dimensional percolation and oriented percolation model, respectively and $F^-$ is the infimum of the support of $F$. Note that if $F$ is continuous, i.e., $\P(\tau_e=a)=0$ for any $a\in \R$, then $F$ is useful.
   \subsection{Main results}
   \begin{thm}\label{thm:shape}
   Suppose that $F$ is useful and there exists $\alpha>0$ such that $\E e^{\alpha \tau_e}<\infty$. For any $\theta\in\mathbb{S}^{d-1}=\{x\in\R^d|~|x|=1\}$ and $r>0$, there exist $c,C>0$ such that for any $t>0$ and closed convex set $\Gamma\subset\R^d$ containing $0$,
$$\P({\rm F}_{L(\theta,r)}(B(t),\Gamma)\le c \log{t})\le C\exp{(-t^c)}.$$
   \end{thm}
   
  We can weaken the exponential moment condition as follows:
    \begin{thm}\label{thm:shape2}
   Suppose that $F$ is useful and $\E[\tau_e^{2m}]<\infty$ with $m\in\N$. Then, for any $\theta\in\mathbb{S}^{d-1}$ and $r>0$, there exist $c,C>0$ such that for any $t>0$ and closed convex set $\Gamma\subset\R^d$ containing $0$,
$$\P({\rm F}_{L(\theta,r)}(B(t),\Gamma)\le c \log{t})\le Ct^{-2dm}.$$
    \end{thm}
  
  \begin{remark}
    Since $\B_d$ is convex and contains $0$, the main result holds for $\Gamma=t\B_d$.
    \begin{remark}
      In fact, the above theorem holds even for a shrinking cone. More precisely, one can see from the proofs below that the following holds: under the condition of Theorem \ref{thm:shape2}, there exists $c>0$ such that for any increasing function $r:(0,\infty)\to (0,\infty)$ with $r(t)\uparrow \infty$ as $t\to\infty$ and $r(t)\le t$,
      \begin{equation}
        \lim_{t\to\infty}\max_{\Gamma}\max_{x\in\partial \Gamma}\P( {\rm F}_{L(x,r(t))}(B(t),\Gamma)\le c\log{r(t)})=0,
        \end{equation}
      where $\Gamma$ runs over all closed convex sets containing $0$. This implies that the fluctuation divereges in any fixed direction.
      \end{remark}
  \end{remark}
  
\subsection{Notation and terminology}
This subsection collects useful notations and terminologies for the proof.
\begin{itemize}
 \item Given two vertices $v,w\in\Z^d$ and a set $D\subset\Z^d$, we set the {\em restricted} first passage time as
$$T_D(v,w)=\inf_{\gamma\subset D}T(\gamma),$$
   where the infimum is taken over all paths $\gamma$ from $v$ to $w$ and $\gamma\subset D$. If such a path does not exist, we set it to be the infinity instead.
   \item Let us define the length of $\gamma=(x_i)^l_{i=1}$ as $\sharp\gamma=l$.
\item It is useful to extend the definition of Euclidean distance $d(\cdot,\cdot)$ as
  $$d(A,B)=\inf\{d(x,y)|~x\in A,~y\in B\}\text{\hspace{4mm}for $A,B\subset \R^d$}.$$
  When $A=\{x\}$, we write $d(x,B)$.
  
\item Given a set $D\subset\Z^d$, let us define the inner boundary of $D$ as $$\partial^- D=\{v\in D|~\exists w\notin D\text{ such that }|v-w|_1=1\}.$$
\item In the proof, we often modify the configuration $\tau$ on a given path $\gamma$. We denote the modified configuration by $\tau^{(\gamma)}=\{\tau^{(\gamma)}_e\}_{e\in E^d}$ and the corresponding first passage time by $T^{(\gamma)}(v,w)$.
\item Let $F^-$ and $F^+$ be the infimum and supremum of the support of $F$, respectively:
  $$F^-=\inf\{\delta\ge 0|~\P(\tau_e<\delta)>0\},~F^+=\sup\{\delta\ge 0|~\P(\tau_e>\delta)>0\}.$$
\end{itemize}
\subsection{Heuristics behind the proof}

\quad Let us briefly explain the basic idea of the proof. One might notice a similarity with the multi-valued map principle, which is, for example, used in \cite{CKNPS}. But, in order to deal with continuous distributions, we use a resampling argument instead in the proof. For simplicity, we suppose that $F(0)>0$ and only discuss how to show that the probability in Theorem~\ref{thm:shape} goes to zero when $L(\theta,r)=\R^d$.\\

%We assume that the following probability does not go to zero:
%\begin{equation}\label{sketch-prob}
%  P({\rm F}_{\R^d} (B(t), \Gamma) \leq c \log t).
%\end{equation}

First, we take $m$ $(:=[t^{1/2}])$ disjoint paths $(\gamma_i)^{m}_{i=1}$ from $\Gamma^-_{c\log{t}}$ to $(\Gamma^+_{c\log{t}})^c$ whose lengths $l_i\in\N$ are at most $2cd\log{t}$ (see Figure~\ref{fig:shape-fluc}). We write $\gamma_i=(\gamma_i[j])^{l_i}_{j=1}$.  Let us denote by $A_i$ the event that $T(0,\gamma_j[l_j])> t$ for any $j\neq i$ and $T(0,\gamma_i[l_i])\leq t$. Note that the $A_i$'s are disjoint events by construction.\\

 We fix a path $\gamma_i$ defined above arbitrarily. We start with the event $\{{\rm F}_{\R^d}(B(t),\Gamma)\le c\log t\}$. On this event, we resample all configurations along $\gamma_i$ and we consider the event that to each edge $e\in\gamma_i$, $\tau_e=0$ after resampling. If $\gamma_j$'s are far enough from each other, it is natural to expect that this resampling does not change the passage times $(T(0,\gamma_j[l_j]))_{j\neq i}$, though the actual proof needs more technical work. Thus $A_i$ holds after resampling. Hence, we should have
$$\P(A_i)\geq \P(\tau_e=0)^{\sharp \gamma_i}\P({\rm F}_{\R^d}(B(t),\Gamma)\le c\log t).$$

 %Obviously, on $A_i$, the shape fluctuation is at least $c\log{t}$ after resampling.
\noindent  By using the facts $A_i$'s are disjoint and $\sharp \gamma_i\leq 2cd\log{t}$, this yields
$$1\geq \sum_{i}\P(A_i)\geq m\P(\tau_e=0)^{2cd\log{t}}\P({\rm F}_{\R^d}(B(t),\Gamma)\le c\log{t}).$$
 Recall that $m=[t^{1/2}]$. Then, using $m\P(\tau_e=0)^{2cd\log{t}}\to\infty$ for sufficiently small $c>0$, we conclude that
 $$\text{$\P({\rm F}_{\R^d}(B(t),\Gamma)\le c\log{t})\to 0$ as $t\to \infty$.}$$
   In order to get the stretched exponential bound in Theorem~\ref{thm:shape}, we apply this argument inductively.
%Let us pretend as if the passage times of $\gamma_i$'s were independent conditi%onally on the event $\{{\rm F}_{\R^d}(B(t),\Gamma)\le c\log t\}$.
%Then it is expected that one of the passage times of $\gamma_i's$ is exactly zero with high probability on this event since $(1- F(0)^{2c\log{t}})^m$ is very small if we take $t$ sufficiently large and $c > 0$ sufficiently small. Note that if ${\rm F}_{\R^d}(B(t),\Gamma)\le c\log t$, then  $T(0,\gamma_i[1])\leq t$, which, together with $T(\gamma_i)=0$, gives $T(0,\gamma_i[l_i])\leq t$ and ${\rm F}_{\R^d}(B(t),\Gamma)> c\log t$. Therefore, \eqref{sketch-prob} is small.\\

%The above strategy does not actually work because the passage times are not independent conditionally on the event $\{{\rm F}_{\R^d}(B(t),\Gamma)\le c\log t\}$.
%A variant of van den Berg-Kesten's resampling argument allows us to convert it into a rigorous argument.  Let us briefly explain how the resampling argument works in the proof.  %Although one construct many maps from the event to evaluate in the multi-valued map principle,\\

    \section{Proof of Theorem \ref{thm:shape}}
We begin with a basic property of convex sets.
  \begin{lem}\label{convex:full}
    Given a convex set $B\subset\R^d$, for any $\delta>0$, $B=(B^+_{\delta})^-_{\delta}.$
  \end{lem}
 \begin{proof}
    If $x\in B$, there exists $\delta>0$ such that $B(x,\delta)\subset
B^+_{\delta}$. Thus $d(x,(B^+_{\delta})^c)\ge \delta$, which implies
$x\in (B^+_{\delta})^-_{\delta}.$ It follows that $ B\subset
(B^+_{\delta})^-_{\delta}$. On the other hand, if $x\notin B$, then
there exists $y\in (B^+_{\delta})^c$ such that $d(x,y)<\delta$, which
implies $x\notin(B^+_{\delta})^-_{\delta}.$ It follows that $ B\supset
(B^+_{\delta})^-_{\delta}$.
    \end{proof}
  Set $\theta\in\mathbb{S}^{d-1}$ and $r>0$.
  \begin{lem}
    There exists $D>0$ such that for sufficiently large $t>0$,
    $$\P\left(\frac{t}{2}\B_d\subset B(t)\right)\geq 1-\exp{(-D t)}.$$
  \end{lem}
  \begin{proof}
     From Theorem 3.13 of \cite{ADH}, there exists $D>0$ such that for any $t>1$ and $x\in \frac{t}{2}\B_d\cap \Z^d$,
    \begin{equation}\label{large-dev}
      \begin{split}
        \P(T(0,x)> t)\le \exp{(-2Dt)}.
    \end{split}
    \end{equation}
    Since $\frac{t}{2}\B_d\not\subset B(t)$ implies that there exists $x\in \frac{t}{2}\B_d$ such that $T(0,x)> t$. Therefore, we obtain
    \begin{equation}\label{large-dev}
      \begin{split}
       \P\left(\frac{t}{2}\B_d\not\subset B(t)\right) \leq \sum_{x\in \frac{t}{2}\B_d\cap \Z^d}\P(T(0,x)> t)\le Ct^d \exp{(-2Dt)}\leq C \exp{(-Dt)}.
    \end{split}
    \end{equation}
  \end{proof}

\begin{figure}[b]\label{fig:intersection}
  \scalebox{0.6}{\includegraphics[width=8.0cm]{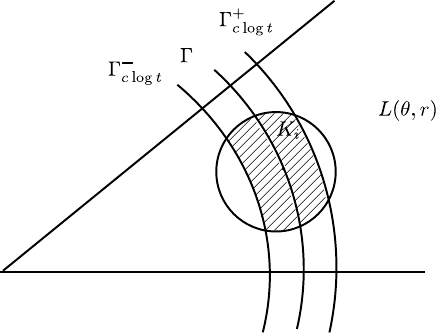}}
  \hspace{12mm}
  \scalebox{0.9}{\includegraphics[width=8.0cm]{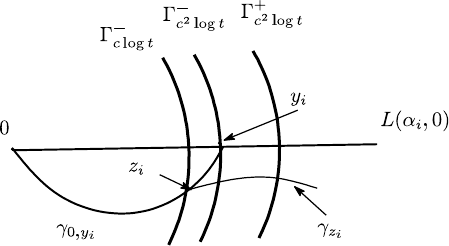}}
\caption{}
Schematic picture of the proof of Theorem \ref{thm:shape}.
  \label{fig:two}
\end{figure}  
 We first consider the case $\frac{t}{3}\B_d \cap L(\theta,r/2)\not \subset \Gamma$. Then, it follows from Lemma \ref{convex:full} that $\frac{t}{2}\B_d \cap L(\theta,r) \not\subset \Gamma^+_{\log{t}}$. If $B(t) \cap L(\theta,r) \subset \Gamma^+_{\log{t}}$, then $\frac{t}{2}\B_d\not\subset B(t)$, which implies
 \begin{equation}\label{simply-case}
   \P({\rm F}_{L(\theta,r)}(B(t),\Gamma)\le \log{t})\le \P\left(\frac{t}{2}\B_d\not \subset B(t)\right)\le \exp{(-Dt)}.
   \end{equation}
 \noindent  Thus without loss of generality, we can restrict ourselves to $\Gamma$'s, which satisfy
  \begin{equation}\label{condition-shape}
  \frac{t}{3}\B_d \cap L(\theta,r/2)\subset \Gamma.
  \end{equation}
  Take a positive constant $\e$ less than $1/2$ arbitrarily. Hereafter, we sometimes omit $[ \cdot ]$ and simply write $t^{\e}$ instead of $[ t^{\e}]$ with some abuse of notation. By \eqref{condition-shape}, for any sufficiently large $t$, there exist $\alpha_1,\cdots,\alpha_{t^{\e}}\in \mathbb{S}^{d-1}\cap  L(\theta,r/4)$ such that for any $i\neq j$,
  \begin{equation}\label{apart}
    d(\partial \Gamma\cap L(\alpha_i,0),\partial \Gamma\cap L(\alpha_j,0))\ge t^{\e}.
  \end{equation}
  Note that $L(\alpha_i,0)\cap \partial \Gamma^-_{c^2\log{t}}$ is a single point since $\Gamma$ is a convex set. For any $i$, we set
  \begin{equation}\label{yi}
    y_i=y_{\alpha_i} =[L(\alpha_i,0)\cap \partial \Gamma^-_{c^2\log{t}}]\in \Z^d.
  \end{equation}

  We use the following property of useful distributions.
\begin{lem}\label{useful}
If $F$ is useful, there exist $\delta>0$ and $D>0$ such that for any $v,w\in\Z^d$,
$$\P(T(v,w)<(F^-+\delta)|v-w|_1)\leq{}e^{-D|v-w|_1}.$$
\end{lem}
For a proof of this lemma, see Lemma 5.5 in \cite{BK93}. We fix $\delta>0$ in Lemma \ref{useful}.
  \begin{Def}
    An $\alpha\in\mathbb{S}^{d-1}$ is said to be black if the following hold:\\
    (1) for any two vertices $v,w\in B(y_{\alpha},2d(\log{t})^2)\cap \Z^d$ with $|v-w|_1\ge \sqrt{\log{t}}$ and a path $\pi:v\to w\subset B(y_{\alpha},2d(\log{t})^2)$, 
    $$T(\pi)\ge (F^-+\delta)|v-w|_1,$$
    (2) for any $e\in E^d$ with $e\subset B(y_{\alpha},2d(\log{t})^2)$,
    $$\tau_e\leq (\log{t})^{2d}.$$
  \end{Def}
  We state the following lemma with a slightly general moment condition to use in the proof of Theorem \ref{thm:shape2}.
  \begin{lem}\label{black-prob}
    If $\E\tau_e^2<\infty$, 
  $$\lim_{t\to\infty}\inf_{\alpha \in\mathbb{S}^{d-1}}\P(\alpha\text{ is black })=1.$$
  \end{lem}
  \begin{proof}
   Note that  there exists $C>0$ independent of $t$ and $\alpha$ such that $$\sharp \{e\in E^d|~e\subset B(y_{\alpha},2d(\log{t})^2)\}\le C (\log{t})^{2d}.$$ By Lemmma \ref{useful} and the union bound, we have
    \begin{equation}
       \begin{split}
        & \P(\{\text{$\alpha$ is black}\}^c)\\
         &\leq \underset{|v-w|_1\ge \sqrt{\log{t}}}{\sum_{v,w\in B(y_{\alpha},2d(\log{t})^2)\cap \Z^d}}\P(T(v,w)\leq (F^-+\delta)|v-w|_1)+\sum_{e\subset B(y_{\alpha},2d(\log{t})^2)}\P(\tau_e>(\log{t})^{2d})\\
         & \le C^2 (\log{t})^{4d} \E[\tau_e^2](\log{t})^{-2d}+ C (\log{t})^{2d} e^{-D\sqrt{\log{t}}}.
         \end{split}
    \end{equation}
    The last term goes to $0$ as $t\to\infty$ uniformly in $\alpha$ and thus we have completed the proof.
    \end{proof}

  \begin{Def}
    ~\\
    \noindent (1) Let $W_1$ be the event  that for any $v,w\in [-t^2,t^2]\cap\Z^d$ with $|v-w|_1\ge t^{\e /2}$, $$T(v,w)\ge (F^-+\delta)|v-w|_1.$$
    \noindent (2) Let $W_2$ be the event  that
    $$\sharp \{i\in\{1,\cdots,t^{\e}\}|~\alpha_i\text{ is black }\}\ge \frac{1}{2}t^{\e}.$$
  \noindent (3) Denote the intersection of $W_1$ and $W_2$ as $W=W_1\cap W_2$.
  \end{Def}
  \begin{lem}\label{exp-decay}
    There exist $c_1,c_2>0$ such that for any sufficiently large $t>0$ 
    $$\P(W^c)\le c_1\exp{(-t^{c_2})}.$$
  \end{lem}
  \begin{proof}
    It is easy to check from Lemma \ref{useful} that  $\P(W_1^c)\leq c_1\exp{(-t^{c_2})}$ with some constants $c_1,c_2>0$. Note that $\mathbf{1}_{\{\text{$\alpha_i$ is black}\}}$ depends only on the configurations on $B(y_i,2d(\log{t})^2)$ and $B(y_i,2d(\log{t})^2)\cap B(y_j,2d(\log{t})^2)=\emptyset$ if $i\neq j$. Therefore $\mathbf{1}_{\{\text{$\alpha_i$ is black}\}}$ and $\mathbf{1}_{\{\text{$\alpha_i\j$ is black}\}}$ are independent if $i\neq j$, which easily yields
    $\P(W_2^c)\leq c_1\exp{(-t^{c_2})}$ by Lemma \ref{black-prob}.
    \end{proof}
  \begin{Def}
    
    We say that $\alpha_i$ is good if ${\rm F}_{L(y_i,(\log{t})^4)}(B(t),\Gamma)> c^2 \log{t}.$ Otherwise, we say that $\alpha_i$ is bad. Given $I\subset \{1,\cdots,t^{\e}\}$, we define an event $A_I$ as
  $$A_I=\{I=\{i\in \{1,\cdots,t^{\e}\}|~\text{$\alpha_i$ is good }\}\}.$$
  \end{Def}
  The reason why we have used $(\log{t})^4$ is just $(\log{t})^2\ll (\log{t})^4\ll t^\e$ and this specific choice is not important.
  \begin{lem}\label{exist-x}
    Let $K_i=(\Gamma^+_{c\log{t}}\backslash \Gamma^-_{ c\log{t}})\cap B(y_i,(\log{t})^2)$. For sufficiently large $t$ depending on $c$, if $\alpha_i$ is bad and black, there exists $x\in \partial^-[K_i\cap\Z^d]$ such that $T_{K_i^c\cup\{x\}}(0,x)\leq t-|x-y_i|_1(F^-+\delta),$ where $K_i^c$ is the complement of $K_i$ in $\R^d$.
  \end{lem}
  
  \begin{proof}
    Take an arbitrary optimal path $\gamma_{0,y_i}=\{x_i\}^{l}_{i=1}$ from $0$ to $y_i$. Since $\alpha_i$ is bad, we have ${\rm F}_{L(y_i,(\log{t})^4)}(B(t),\Gamma)\le c^2 \log{t}$, which implies $t(0,y_i)\le t$. Let $x$ be the first intersecting point of $\gamma_{0,y_i}$ and $\partial^-[K_i\cap\Z^d]$, i.e., $$\text{$m=\min \{i\in\{1,\cdots,l\}|~x_i \in \partial^-[K_i\cap\Z^d]\neq \emptyset\}$ and $x=x_m$.}$$ Since $|x-y_i|_1\ge \frac{c}{2}\log{t}$ and $\alpha_i$ is black, we have
    $$T_{K_i^c\cup\{x\}}(0,x)\leq t-|x-y_i|_1(F^-+\delta).$$
    \end{proof}
  
  \begin{lem}\label{crucial-shape}
   Let $A^k=\bigcup_{\sharp I=k} A_I$, where the union runs over all subsets $I\subset \{0,\cdots,n\}$ with $\sharp I=k$. Then, for any $k\in\{0,\cdots,t^{\e/2}\}$,    
    \begin{equation}
      \begin{split}
       t^{\e/8} \P\left(A^k \cap W\right)\leq \P(A^{k+1}).
      \end{split}
    \end{equation}
  \end{lem}
  We postpone the proof of this lemma and first complete the proof of Theorem \ref{thm:shape}.
  \begin{proof}[Proof of Theorem \ref{thm:shape}]
    Combining the previous lemma with Lemma \ref{exp-decay}, we have that for any $k,l<t^{\e/2}/2$,
    \begin{equation}
      \begin{split}
        \P\left(A^{k+l}\right)&\geq t^{\e/8}(\P(A^{k+l-1})-\P(W^c))\\
        & \geq t^{\e/8}\P(A^{k+l-1})-c_1t^{\e/8}\exp{(-t^{c_2})}\\
        &\geq t^{2\e/8}\P(A^{k+l-2})-c_1\sum^2_{n=1}t^{n\e/8}\exp{(-t^{c_2})}.
    \end{split}
    \end{equation}
    Continuing this procedure, for sufficiently large $t>0$, if $k\le t^{\e/2}/2$ and $l\le t^{c_2/2}$, we have
    \begin{equation}
      \begin{split}
        \P\left(A^{k+l}\right)&\geq t^{l\e/8}\P(A^{k})-c_1\sum^l_{n=1}t^{n\e/8}\exp{(-t^{c_2})}\\
        &\geq t^{l\e/8}\P(A^{k})-c_1 t^{l\e/4}\exp{(-t^{c_2})}\\
        &\geq t^{l\e/8}\P(A^{k})-c_1 \exp{(-t^{c_2}/2)}.
    \end{split}
    \end{equation}
   Applying it with $k=0$ and $l=t^{c_2/2}$ yields
        \begin{equation}
      \begin{split}
        \P\left(A^0\right)&\leq \P(A^{t^{c_2/2}})t^{-\e t^{c_2/2}}+c_1\exp{(-t^{c_2}/2)}\\
        &\leq 2c_1\exp{(-t^{c_2}/2)}.
    \end{split}
    \end{equation}
        Since ${\rm F}_{L(\theta,r)}(B(t),\Gamma)\le c^2 \log{t}$ implies that $A^0$ occurs, it follows that
        $$\P({\rm F}_{L(\theta,r)}(B(t),\Gamma)\le c^2 \log{t})\le 2c_1\exp{(-t^{c_2/2})},$$
        as desired.
        \end{proof}
  \begin{proof}[Proof of Lemma \ref{crucial-shape}]
    
  Given $z\in \partial^-[K_i\cap\Z^d]$, we take ${\rm End}(z)\in K_i\cap (\Gamma^+_{c^2\log{t}})^c \cap \Z^d$  and a path $\gamma_z:z\to {\rm End}(z)$  such that  $\gamma_z\subset K_i$ and $\sharp \gamma_z\le (|y_i-z|_1+2dc^{2}\log{t})\land 2dc \log{t}$ with a deterministic rule breaking ties. Let $z_i$ be a random variable uniformly distributed on $\partial^-[K_i\cap\Z^d]$ which is independent of $\tau=\{\tau_e\}_{e\in E^d}$. Let $(\tilde{P}_i,\tilde{\Omega}_i)$ be its probability space. We simply write $\gamma_i$ for $\gamma_{z_i}$ hereafter.\\
          
          Let $\tau^*=\{\tau^*_e\}_{e\in E^d}$ be an independent copy of $\{\tau_e\}_{e\in E^d}$ and also independent of $z_i$. We enlarge the probability space so that it can measure the events both for $\tau$ and $\tau^*$ and we still denote the joint probability measure by $\P$.  Given a path $\gamma$, we define the resampled configuration $\tau^{(\gamma)}=\{\tau^{(\gamma)}_e\}_{e\in E^d}$ as
  $$\tau^{(\gamma)}_e=\begin{cases}
    \tau_{e}^* & \text{if $e\in\gamma$}\\
  \tau_{e} & \text{otherwise.}
  \end{cases}$$
  Note that the distributions of $\tau$ and $\tau^{(\gamma_{i})}$ are the same under $\tilde{P}_i\otimes\mathbb{P}$ since $\tau,~\tau^*,~z_i$ are independent.\\

  Let $I\subset \{1,\cdots,t^{\e}\}$ be such that $\sharp I=k$ and take $i\notin I$. Given $z\in\partial^-[K_i\cap \Z^d]$, we set
  $$\mathcal{G}(z)=\{T_{K_i^c\cup\{z\}}(0,z)\leq t-|z-y_i|_1(F^-+\delta)\}\text{ and }\mathcal{S}(z)=\{\forall e\in\gamma_{z},~\tau^*_{e}\leq F^-+\delta/2\}.$$ Then we define an event as
  $$\tilde{A}_{I,i}=A_I\cap W\cap\mathcal
  {G}(z_i)\cap \mathcal{S}(z_i)\cap\{\text{$\alpha_i$ is black for $\tau$}\}.$$
  We will show that $\tilde{A}_{I,i}$ implies
  \begin{equation}\label{discrim}
    I\cup \{i\}=I^{(\gamma_i)},
    \end{equation}
  where $I^{(\gamma_i)}=\{i\in \{1,\cdots,t^\e\}|~\text{$\alpha_i$ is good for $\tau^{(\gamma_i)}$}\}\}$.
    Under the conditions $\mathcal{G}(z_i)$ and $\mathcal{S}(z_i)$, by the construction of $\gamma_{z}$, we have
  \begin{equation*}
      \begin{split}
        T^{(\gamma_i)}(0,{\rm End}(z_i))&\le T^{(\gamma_i)}(0,z_i)+T^{(\gamma_i)}(\gamma_i)\\
        &\leq t-|z_i-y_i|(F^-+\delta)+(F^-+\delta)|y_i-z_i|_1=t.
    \end{split}
  \end{equation*}
  Thus, $\alpha_i$ is good for $\tau^{(\gamma_i)}$. On the other hand, if
  $\mathbf{1}_{\{\alpha_j\text{ is good for }\tau\}}\neq \mathbf{1}_{\{\text{$\alpha_j$ is good for }\tau^{(\gamma_i)}\}}$ for some $j\neq i$, then there exist $w\in K_j$ and a path $\Gamma=(x_i)^l_{i=1}:0\to w$ with $\Gamma\cap K_i\neq\emptyset$ such that $T(\Gamma)\le t$ or $T^{(\gamma_i)}(\Gamma)\le t$. Indeed any $w\in (B(t) \cap K_j) \bigtriangleup (B^{(\gamma_i)}(t) \cap K_j)$  has such property, where $A\bigtriangleup B$ is the symmetric difference of $A$ and $B$. Note that $(B(t) \cap K_j) \bigtriangleup (B^{(\gamma_i)}(t) \cap K_j)$ is nonempty exactly because of the condition $\mathbf{1}_{\{\alpha_j\text{ is good for }\tau\}}\neq \mathbf{1}_{\{\text{$\alpha_j$ is good for }\tau^{(\gamma_i)}\}}$. Let $m=\min \{i\in\{1,\cdots,l\}|~x_i \in\partial^-[K_i\cap\Z^d]\neq \emptyset\}$ and $x=x_m$. Then under the condition $W$, by \eqref{apart}, we have $$T_{K_i^c\cup\{x\}}(0,x)\leq t-t^{\e/2}.$$
  Since $\alpha_i$ is black (in particular for any $e\in E^d$ with $e\subset B(y_i,2d(\log{t})^2)$,  $\tau_e\leq (\log{t})^{2d}$) and there exists a path $\gamma\subset K_i$ from $x$ to some $\tilde{x}\in (\Gamma^+_{c\log{t}})^c$ whose length is at most $2cd\log{t}$, we obtain $T(0,\tilde{x})\le t-t^\e+2cd(\log{t})^{2d+1}\leq t$. Therefore $\alpha_i$ is good for $\tau$, which contradicts that $i\notin I$. Therefore we have $\mathbf{1}_{\{\alpha_j\text{ is good for }\tau\}}= \mathbf{1}_{\{\text{$\alpha_j$ is good for }\tau^{(\gamma_i)}\}}$ and \eqref{discrim} follows.\\

  From this observation, we have
       \begin{equation}\label{2.7}
      \begin{split}
        \P(A^{k+1})&=\sum_{\sharp I=k+1}\P(A_I)\\
        &=\frac{1}{k+1}\sum_{\sharp I=k}\sum_{i\notin I}\P(A_{I\cup \{i\}})\\
        &=\frac{1}{k+1}\sum_{\sharp I=k}\sum_{i\notin I}\P(A_{I\cup \{i\}}^{(\gamma_i)})\\
        &\ge \frac{1}{k+1}\sum_{\sharp I=k}\sum_{i\notin I}\tilde{P}_i\otimes \P(\tilde{A}_{I,i}).
          \end{split}
    \end{equation}
       Since $\tau$, $\tau^*$ and $z_i$ are independent, $\tilde{P}_i\otimes \P(\tilde{A}_{I,i})$ can be bounded from below as

       \begin{equation}\label{sports}
         \begin{split}
           \tilde{P}_i\otimes \P(\tilde{A}_{I,i})  &=\frac{1}{\sharp \partial^-[K_i\cap \Z^d]}\sum_{z\in \partial^-[K_i\cap \Z^d]}\P(A_I\cap W\cap\mathcal{G}(z)\cap \mathcal{S}(z)\cap\{\text{$\alpha_i$ is black for $\tau$}\})\\
           &= \frac{1}{\sharp \partial^-[K_i\cap \Z^d]}\sum_{z\in \partial^-[K_i\cap \Z^d]}\P(A_I\cap W\cap\mathcal{G}(z)\cap\{\text{$\alpha_i$ is black for $\tau$}\})\P(\mathcal{S}(z))\\
         &\ge \frac{\min_{z\in \partial^-[K_i\cap \Z^d]}\P(\mathcal{S}(z))}{\sharp \partial^-[K_i\cap \Z^d]}  \E[\sharp\{z\in \partial^-[K_i\cap \Z^d]|~\mathcal{G}(z)\}; A_I\cap W\cap\{\text{$\alpha_i$ is black for $\tau$}\}].\\
         \end{split}
       \end{equation}
       Lemma~\ref{exist-x} implies $\sharp\{z\in \partial^-[K_i\cap \Z^d]|~\mathcal{G}(z)\}\geq 1$ on the event $A_I\cap W\cap \{\text{$\alpha_i$ is black for $\tau$}\}$. Combining it with the condition that $|\gamma_{z}|\le 2dc\log{t}$ for any $z\in\partial^-[K_i\cap\Z^d]$, \eqref{sports} is bounded from below by
        \begin{equation}
         \begin{split}
        \frac{1}{\sharp \partial^-[K_i\cap \Z^d]}\P(\tau_e<F^-+\delta/2)^{2dc\log{t}} \P(A_I\cap W\cap \{\text{$\alpha_i$ is black for $\tau$}\}).
        \end{split}
       \end{equation}
           
       Thus  if $c$ is sufficiently small depending on $\e$ and $\delta$, \eqref{2.7} is bounded from below by 
        
       \begin{equation}
      \begin{split}
       &\frac{1}{k+1}\frac{1}{\sharp \partial^-[K_i\cap \Z^d]}\P(\tau_e<F^-+\delta/2)^{2dc\log{t}}\sum_{\sharp I=k}\sum_{i\notin I}\P(A_I\cap W\cap\{\text{$\alpha_i$ is black for $\tau$}\})\\
        & \ge t^{-3\e /4}\sum_{\sharp I=k}\E[\sharp\{i\notin I|~\text{$\alpha_i$ is black for $\tau$}\}; A_I\cap W].
        \end{split}
       \end{equation}
       Since $\sharp\{i\notin I|~\text{$\alpha_i$ is black for $\tau$}\}\ge t^\e/2-k$ on the event $A_I\cap W$, this is further bounded from below by
       \begin{equation}
         t^{-3\e /4} \left(\frac{t^\e}{2}-k\right)\P(A^k\cap W)\ge t^{\e/8}\P(A^k\cap W),\end{equation} as desired.
  \end{proof}
        \section{Proof of Theorem \ref{thm:shape2}}
        We take $K>0$ so that $2\E[\tau_e](|x|_1+8)< t$ for any $x\in \frac{t}{K}\B_d$ and $t>1$. We first consider the case
        $$\frac{t}{2K}\B_d \cap L(\theta,r/2)\subset \Gamma.$$ Comparing this case with \eqref{condition-shape}, since \eqref{condition-shape} was used only to get \eqref{apart} and the finite exponential moment condition was used only in \eqref{simply-case} and Lemma~\ref{black-prob}, the exactly same proof works.\\

        Next, we suppose that $\frac{t}{2K}\B_d \cap L(\theta,r/2)\not\subset \Gamma$. In the proof of Theorem~\ref{thm:shape}, we use the finite exponential moment condition to prove  \eqref{large-dev}. We modify \eqref{large-dev} as follows. By Lemma \ref{convex:full}, if $t$ is sufficiently large, $ \left[\frac{t}{K}\B_d \cap L(\theta,r) \backslash \Gamma^{+}_{\log{t}}\right]\cap\Z^d$ is non-empty and we take an arbitrary vertex $x$ of this set. If $T(0,x)\le t$, then $x\in B(t)$, which yields that $B(t) \cap L(\theta,r) \not\subset \Gamma^{+}_{\log{t}}$. Now we consider $2d$ disjoint paths $\{r_i\}^{2d}_{i=1}$ from $0$ to $x$ so that $$\max\{|r_i||~i=1,\cdots,2d\}\leq |x|_1+8$$ as in \cite[p 135]{Kes86}.  Then, it follows from the Chebyshev inequality that there exists $C>0$ such that  
   \begin{equation}\label{fell}
     \begin{split}
       \P\left( T(0,x)> t\right)& \leq \prod^{2d}_{i=1} \P\left(T(r_i)> t\right)\\
       &\leq \prod^{2d}_{i=1} \P\left(|T(r_i)-\E[T(r_i)]|> t/2\right)\\
       &\leq \prod^{2d}_{i=1} \left((t/2)^{-2m}(|x|_1+8)^m \E[\tau_e^{2m}]\right) \leq Ct^{-2dm}. 
  \end{split}
   \end{equation}
   It yields that
   \begin{equation}\label{conclusion2}
     \begin{split}
       \P({\rm F}_{L(\theta,r)}(B(t),\Gamma)\le \log{t})\le \P(B(t) \cap L(\theta,r) \subset \Gamma^{+}_{\log{t}})\\
       \le \P(T(0,x)>t)\le Ct^{-2dm}.
  \end{split}
   \end{equation}
%    \begin{comment}
\section*{Acknowledgements}
 The author would like to express his gratitude to Yohsuke T. Fukai for useful comments on the theoretical and experimental researches of the shape fluctuation in physics. Thanks also go to Masato Takei for introducing him the idea of Theorem~2 in \cite{Zhang06}. This research is partially supported by JSPS KAKENHI 16J04042.
%\end{comment}

\end{document}